\documentclass[10pt]{article}
\usepackage{etex}
\usepackage{graphics}
\usepackage{graphicx}
\usepackage{amsmath}
\usepackage{amsfonts}
\usepackage{amssymb}
\usepackage{amsthm}
\usepackage{latexsym}
\usepackage{lpic}
\usepackage{hyperref}
   \usepackage{epstopdf}
\usepackage{amsmath}
\usepackage{amssymb}
\usepackage{amsfonts}
\usepackage{amsbsy}
\usepackage{subfigure}
\usepackage{tabularx,multicol,multirow}
\usepackage{bbm}
\usepackage[]{algorithm2e}
\usepackage{fullpage}
\usepackage[utf8]{inputenc}
\usepackage{pgf}
\usepackage{tikz}
\usetikzlibrary{shapes,arrows,spy,positioning,decorations,calc}
\tikzset{
insep/.style={inner sep=2pt, outer sep=0pt, circle,fill},
free/.style={inner sep=2pt, outer sep=0pt, circle,fill=gray,draw},
extra/.style={inner sep=2pt, outer sep=0pt, circle,fill=white,draw},
}
\usepackage[sanserif]{complexity}

\newtheorem{theorem}{Theorem}
\newtheorem{lemma}[theorem]{Lemma}
\newtheorem{claim}{Claim}[theorem]

\newtheorem{prop}[theorem]{Proposition}

\theoremstyle{definition}
\newtheorem*{remark}{Remark}
\newtheorem*{defn}{Definition}

\def\dal{\operatorname{dal}}

\begin{document}

\title{Color-blind index in graphs of very low degree\thanks{This collaboration began as part of the 2014 Rocky Mountain--Great Plains Graduate Research Workshop in Combinatorics, supported in part by NSF Grant \#1427526.}}

\author{Jennifer Diemunsch$^{1}$ \and Nathan Graber$^{1}$ \and Lucas Kramer$^{2}$ \and Victor Larsen$^{3}$ \and Lauren M. Nelsen$^{4}$ \and Luke L. Nelsen$^{1}$ \and Devon Sigler$^{1}$ \and Derrick Stolee$^{5}$ \and Charlie Suer$^{6}$}

\maketitle
\footnotetext[1]{Department of Mathematical and Statistical Sciences, University of Colorado Denver, Denver, CO 80217 ; {\tt $\{$jennifer.diemunsch,nathan.graber,luke.nelsen,devon.sigler$\}$@ucdenver.edu}.}
\footnotetext[2]{Department of Mathematics, Bethel College, North Newton, KS 67117; {\tt lkramer@bethelks.edu}.}
\footnotetext[3]{Department of Mathematics, Kennesaw State University, Marietta, GA 30060; \texttt{vlarsen@kennesaw.edu}.}
\footnotetext[4]{Department of Mathematics, University of Denver, Denver, CO 80208; {\tt lauren.morey@du.edu}.}
\footnotetext[5]{Department of Mathematics, Department of Computer Science, Iowa State University, Ames, IA 50011; {\tt dstolee@iastate.edu}.}
\footnotetext[6]{Department of Mathematics, University of Louisville, Louisville, KY 40292; {\tt cjsuer01@louisville.edu}.}

\begin{abstract}
Let $c:E(G)\to [k]$ be an edge-coloring of a graph $G$, not necessarily proper.
For each vertex $v$, let $\bar{c}(v)=(a_1,\ldots,a_k)$, where $a_i$ is the number of edges incident to $v$ with color $i$. 
Reorder $\bar{c}(v)$ for every $v$ in $G$ in nonincreasing order to obtain $c^*(v)$, the color-blind partition of $v$.
When $c^*$ induces a proper vertex coloring, that is, $c^*(u)\neq c^*(v)$ for every edge $uv$ in $G$, we say that $c$ is color-blind distinguishing.
The minimum $k$ for which there exists a color-blind distinguishing edge coloring $c:E(G)\to [k]$ is the color-blind index of $G$, denoted $\dal(G)$.
We demonstrate that determining the color-blind index is more subtle than previously thought.
In particular, determining if $\dal(G) \leq 2$ is $\NP$-complete.
We also connect the color-blind index of a regular bipartite graph to 2-colorable regular hypergraphs and characterize when $\dal(G)$ is finite for a class of 3-regular graphs.
\end{abstract}

\section{Introduction}

Coloring the vertices or edges of a graph $G$ in order to distinguish neighboring objects is fundamental to graph theory.
While typical coloring problems color the same objects that they aim to distinguish, it is natural to consider how edge-colorings can distinguish neighboring vertices.
For an edge-coloring $c$ using colors $\{1,\dots,k\}$, the \emph{color partition} of a vertex $v$ is given as $\bar{c}(v) = (a_1,\dots,a_k)$, where the integer $a_i$ is the number of edges incident to $v$ with color $i$.
The edge-coloring $c$ is \emph{neighbor distinguishing} if $\bar{c}$ is a proper vertex coloring of the vertices of $G$.
The \emph{neighbor-distinguishing index} of $G$ is the minimum $k$ such that there exists a neighbor distinguishing $k$-edge-coloring of $G$.
Define $c^*(v)$ to be the list $\bar{c}(v)$ in nonincreasing order; call $c^*(v)$ the \emph{color-blind partition} at $v$, since $c^*(v)$ allows for counting the sizes of the color classes incident to $v$ without identifying the colors.
The edge-coloring $c$ is \emph{color-blind distinguishing} if $c^*$ is a proper vertex coloring of the vertices of $G$.
The \emph{color-blind index} of $G$, denoted $\dal(G)$, is the minimum $k$ such that there exists a color-blind distinguishing $k$-edge-coloring of $G$.

The neighbor-distinguishing index and color-blind index do not always exist for a given graph $G$.
A graph $G$ has no neighbor-distinguishing coloring if and only if it contains a component containing a single edge~\cite{balister2003vertex}.
The conditions that guarantee $G$ has a color-blind distinguishing coloring are unclear.
When a graph $G$ has no color-blind distinguishing coloring, we say that $\dal(G)$ is undefined or write $\dal(G) = \infty$.
Kalinowski, Pil{\'s}niak, Przyby{\l}o, and Wo{\'z}niak~\cite{kalinowski2013can} defined color-blind distinguishing colorings and presented several examples of graphs that have no color-blind distinguishing colorings.
All of the known examples that fail to have color-blind distinguishing colorings have minimum degree at most three.

When two adjacent vertices have different degree, their color-blind partitions are distinct for every edge-coloring.
Thus, it appears that constructing a color-blind distinguishing coloring is most difficult when a graph is regular and of small degree.
Most recent work~\cite{achlioptas2014random,przybylo2012colour} has focused on demonstrating that $\dal(G)$ is finite and small when $G$ is a regular graph (or is almost regular) of large degree.
These results were improved by Przyby\l{}o~\cite{przybylo2014colour} in the following theorem.

\begin{theorem}[Przyby\l{}o~\cite{przybylo2014colour}]
If $G$ is a graph with minimum degree $\delta(G) \geq 3462$, then $\dal(G) \leq 3$.
\end{theorem}

We instead focus on graphs with very low minimum degree.
In Section~\ref{sec:hardness}, we demonstrate that it is difficult to determine $\dal(G)$, even when it is promised to exist.

\begin{theorem}\label{thm:mainnphard}
Determining if $\dal(G) = 2$ is $\NP$-complete, even under the promise that $\dal(G) \in \{2,3\}$.
\end{theorem}

The hardness of determining $\dal(G)$ implies that there is no efficient characterization of graphs with low color-blind index (assuming $\P \neq \NP$).
Therefore, we investigate several families of graphs with low degree in order to determine their color-blind index.
For example, it is not difficult to demonstrate that $\dal(G) \leq 2$ when $G$ is a tree on at least three vertices.

A 2-regular graph is a disjoint union of cycles, and the color-blind index of cycles is known~\cite{kalinowski2013can}, so we pursue the next case by considering different classes of 3-regular graphs, and determine if they have finite or infinite color-blind index.
If $G$ is a $k$-regular bipartite graph, then the color-blind index of $G$ is at most 3~\cite{kalinowski2013can}.
In Section~\ref{propB}, we demonstrate that a $k$-regular bipartite graph has color-blind index 2 exactly when it is associated with a 2-colorable $k$-regular $k$-uniform hypergraph.
Thomassen~\cite{thomassen1992even} and Henning and Yeo~\cite{henning20132} proved that all $k$-regular, $k$-uniform hypergraphs are 2-colorable when $k \geq 4$; this demonstrates that all $k$-regular bipartite graphs have color-blind index at most 2 when $k \geq 4$.
Thus, for $k$-regular bipartite graphs it is difficult to distinguish between color-blind index 2 or 3 only when $k = 3$.

To further investigate 3-regular graphs, we consider graphs that are very far from being bipartite in Section~\ref{sec:3-cycles}.
In particular, we consider a connected 3-regular graph $G$ where every vertex is contained in a 3-cycle.
If there is a vertex in three 3-cycles, then $G$ is isomorphic to $K_4$ and there does not exist a color-blind distinguishing coloring of $G$~\cite{kalinowski2013can}.
If $v$ is a vertex in two 3-cycles, then one of the neighbors $u$ of $v$ is in both of those 3-cycles.
These two 3-cycles form a \emph{diamond}.
We say $G$ is a \emph{cycle of diamonds} if $G$ is a 3-regular graph where every vertex in $G$ is in a diamond; $G$ is an \emph{odd cycle of diamonds} if $G$ is a cycle of diamonds and contains $4t$ vertices for an odd integer $t$.
In particular, we consider $K_4$ to be a {cycle of one diamond}.

\begin{theorem}\label{thm:maindiamonds}
Let $G$ be a connected 3-regular graph where every vertex is in at least one 3-cycle of $G$.
$G$ has a color-blind coloring if and only if $G$ is not an odd cycle of diamonds.
When $G$ is not an odd cycle of diamonds, then $\dal(G) \leq 3$.
\end{theorem}


 \section{Hardness of Computing $\dal(G)$}\label{sec:hardness}

In this section, we prove Theorem~\ref{thm:mainnphard} in the standard way.
For basics on computational complexity and $\NP$-completeness, see~\cite{arora2009computational}.
It is clear that a nondeterministic algorithm can produce and check that a coloring is color-blind distinguishing, so determining $\dal(G) \leq k$ is in $\NP$.
We define a polynomial-time reduction\footnote{This reduction could easily be implemented in logspace.} that takes a boolean formula in conjunctive normal form where all clauses have three literals and outputs a graph with color-blind index two if and only if the boolean formula is satisfiable.

\setcounter{theorem}{2}
 \vspace{1em}
\noindent\textbf{Theorem~\ref{thm:mainnphard}.} \textit{
Determining if $\dal(G) = 2$ is $\NP$-complete, even under the promise that $\dal(G) \in \{2,3\}$.}

\begin{proof}
To prove hardness we will demonstrate a polynomial-time reduction that, given an instance $\phi$ of 3-SAT, will produce a graph $G_\phi$ such that $2 \le \dal(G_\phi) \leq 3$ and such that $\dal(G_\phi) = 2$ if and only if $\phi$ is satisfiable.

Let $\phi(x_1,\dots,x_n) = \bigwedge_{i=1}^m C_i$ be a 3-CNF formula with $n$ variables $x_1,\dots,x_n$ and $m$ clauses $C_1, \dots, C_m$.
Let each clause $C_j$ be given as $C_j = \hat{x}_{i_{j,1}} \vee \hat{x}_{i_{j,2}} \vee \hat{x}_{i_{j,3}}$, where each $\hat{x}_{i_{j,k}}$ is one of $x_{i_{j,k}}$ or $\neg x_{i_{j,k}}$.

We will construct a graph $G_\phi$ by creating gadgets that represent each variable and clause, and then identifying vertices within those gadgets in order to create $G_\phi$.
In a 2-edge-coloring of $G_{\phi}$, we consider the color-blind partition $(2,1)$ to be a ``true'' value while the partition $(3,0)$ corresponds to a ``false'' value.

Let $V$ be the graph given by vertices $p_0,p_1,\dots,p_{6m+7}, v_1,\dots,v_{6m+6}, r_1,\dots,r_{12m+12}$ where the vertices $p_0p_1\dots p_{6m+7}$ form a path, and each $v_i$ is adjacent to $p_i, r_{2i-1}$ and $r_{2i}$.
We will call $V$ the \emph{variable gadget} and create a copy $V_i$ of $V$ for each variable $x_i$, and list the copy of each vertex $w$ as $w^i$.
The vertices $p_1,\dots,p_{6m+6}$ and $v_1,\dots,v_{6m+6}$ all have degree three, so in a 2-edge-coloring of $V$, the color-blind partitions take value $(2,1)$ or $(3,0)$.
If the color-blind partitions form a proper vertex coloring, then these partitions alternate along the path $p_1\dots p_{6m+6}$ and along the list $v_1\dots v_{6m+6}$.
Hence, if $G_\phi$ has a color-blind distinguishing 2-edge-coloring, then the color-blind partition of $v_1^i$ in the copy $V_i$ corresponds to the truth value of $x_i$.
If a clause $C_j$ contains the variable $\hat{x}_{i}$, the vertices $v_{6j+3}^i$ and $v_{6j+4}^i$ will be used in order to connect the value of $x_i$ or $\neg x_i$ to the clause.
First, we must discuss the clause gadgets.

Let $L$ be the graph given by a 3-cycle $z_1z_2z_3$, a 14-cycle $u_1u_2\dots u_{14}$, and vertices $\ell_4, \ell_7, \ell_{10}$ with the addition of edges $z_1u_1$, $u_4\ell_4$, $u_7\ell_7$, $u_{10}\ell_{10}$.
See Figure~\ref{fig:clausegadget} for the graph $L$.
For each clause $C_j$, create a copy $L_j$ of $L$ and let $t_1^j$, $t_2^j$, $t_3^j$, $s_1^j$, $s_2^j$ and $s_3^j$ be the copies of the vertices $u_4$, $u_7$, $u_{10}$, $\ell_4$, $\ell_7$ and $\ell_{10}$ in $L_j$.

 \begin{figure}[htp]
 \centering
 \mbox{
 \subfigure[\label{fig:variablegadget}A variable gadget, $V$, where $m=0$.]{
 \begin{tikzpicture}[scale=1]

\draw (0,0) node[insep]{}; \draw (0,0) node[below]{$p_0$};
\draw (0,0)--(1,0);
\foreach \x in {1,2,3,4,5,6}
{
	\draw (\x,0.5)--(\x,0)--(\x+1,0);
	\draw (\x-0.33,1)--(\x,0.5)--(\x+0.33,1);
	\draw (\x,0) node[insep]{}; \draw (\x,0) node[below]{$p_{\x}$};
	\draw (\x,0.5) node[insep]{}; \draw (\x,0.5) node[left]{$v_{\x}$};
	\draw (\x-0.33,1) node[insep]{};
	\draw (\x+0.33,1) node[insep]{};
}  \foreach \x in {1,3,5,7,9,11} {	\draw (\x/2+0.25,1) node[above]{$r_{\x}$}; } \foreach \x in {2,4,6,8,10,12} {	\draw (\x/2+0.25,1) node[above]{$r_{\x}$}; }
\draw (7,0) node[insep]{}; \draw (7,0) node[below]{$p_7$};
 \end{tikzpicture}
 }
 \subfigure[\label{fig:clausegadget}A clause gadget, $L$.]{
 \begin{tikzpicture}[scale=1.5]
 \def\offset{90-360/14}
\foreach \a in {1,2,3,4,5,6,7,8,9,10,11,12,13,14}
{
	\draw (\a*360/14+\offset:1) node[insep]{};
	\draw (\a*360/14+\offset:0.75) node[]{$u_{\a}$};
	\draw (\a*360/14+\offset:1)--(\a*360/14+\offset+360/14:1);
}

\draw (1*360/14+\offset:1)--(90:1.4)--(80:1.75)--(100:1.75)--(90:1.4);
\draw (90:1.4) node[insep]{};
\draw (80:1.75) node[insep]{};
\draw (100:1.75) node[insep]{};
\draw (90:1.4) node[below left]{$z_1$};
\draw (80:1.75) node[above right]{$z_3$};
\draw (100:1.75) node[above left]{$z_2$};

\draw (4*360/14+\offset:1.35) node[insep]{}; \draw (4*360/14+\offset:1.35) node[below]{$\ell_4$};
\draw (7*360/14+\offset:1.35) node[insep]{}; \draw (7*360/14+\offset:1.35) node[right]{$\ell_7$};
\draw (10*360/14+\offset:1.35) node[insep]{}; \draw (10*360/14+\offset:1.35) node[above right]{$\ell_{10}$};
\draw (4*360/14+\offset:1)--(4*360/14+\offset:1.35);
\draw (7*360/14+\offset:1)--(7*360/14+\offset:1.35);
\draw (10*360/14+\offset:1)--(10*360/14+\offset:1.35);
 \end{tikzpicture}
 }
 }
 \caption{\label{fig:gadgets}Gadgets for variables and clauses.}
 \end{figure}
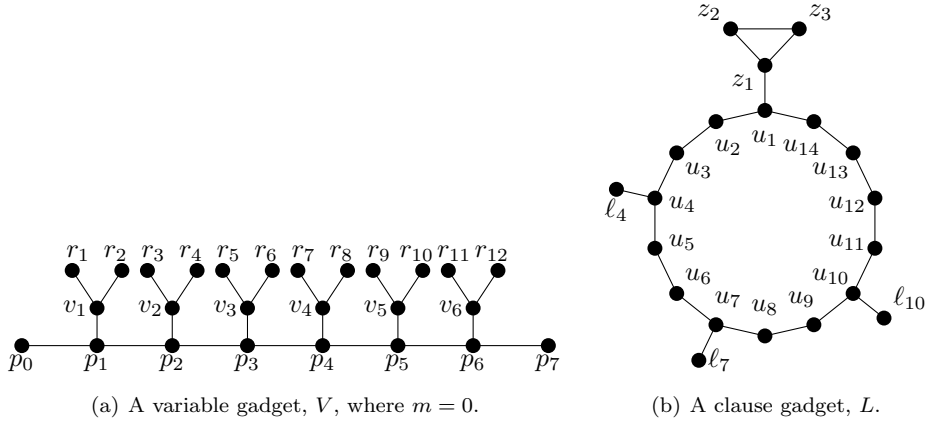

\begin{claim}\label{claim:clausegadgetonetrue}
Let $c$ be a 2-edge-coloring of $L$ and let $c^*$ be the color-blind partitions on the vertices of $L$.
If $c^*$ is a proper vertex coloring, then at least one of the vertices $u_4$, $u_7$, and $u_{10}$ has color-blind partition $(2,1)$.
\end{claim}

\begin{proof}
Suppose for the sake of contradiction that $c^*$ is a proper vertex coloring and the vertices $u_4$, $u_7$, and $u_{10}$ all have color-blind partition $(3,0)$.
Thus, the two edges on the cycle incident to one of these vertices have the same color.

In the cycle $z_1z_2z_3$, the 2-vertices $z_2$ and $z_3$ must have different color-blind partitions.
Thus, the edges $z_1z_2$ and $z_3z_1$ must receive distinct colors $a$ and $b$.
Thus $c^*(z_1) = (2,1)$ and hence $c^*(u_1) = (3,0)$.
Therefore, all 3-vertices on the 14-cycle have the color-blind partition $(3,0)$.

Without loss of generality, let $a$ be the color on the edges $u_1u_2$ and $u_{14}u_1$.
Observe that since the 2-vertices $u_2$ and $u_3$  have distinct color-blind partitions, the edge $u_3u_4$ has color $b$ and hence $u_4u_5$ has color $b$.
Similarly, observe that the edges $u_6u_7$ and $u_7u_8$ have color $a$,
and again that the edges $u_9u_{10}$ and $u_{10}u_{11}$ have color $b$.

Now, the 2-vertices $u_{11}$, $u_{12}$, $u_{13}$, and $u_{14}$ should have distinct color-blind partitions, but since the color of $u_{10}u_{11}$ is $b$ and the color of $u_{14}u_1$ is $a$, this is impossible.
\end{proof}

It remains to show that if at least one of these vertices has color-blind partition $(2,1)$, then we can give a color-blind distinguishing 2-edge-coloring to the gadget $L$.

\begin{claim}\label{claim:clausegadgetsatisfy}
Let $p_4,p_7,p_{10}$ be three partitions in $\{ (2,1), (3,0)\}$.
If at least one value $p_i$ is $(2,1)$, then there exists a 2-edge-coloring $c$ of $L$ such that $c^*$ is a proper vertex coloring and $c^*(u_4) = p_4$, $c^*(u_7) = p_7$, and $c^*(u_{10}) = p_{10}$.
\end{claim}

\begin{proof}
Select $j \in \{4,7,10\}$ such that $p_j = (2,1)$.
We construct a 2-edge-coloring $c$ of $L$ by first coloring the edges $v_1v_2$, $v_2v_3$, $u_1v_1$, and $u_1u_2$ with color $a$ and the edge $v_1v_3$ with color $b$.
We will color the cycle $u_1u_2\dots u_{14}$ by coloring the edges of the paths $u_1u_2u_3u_4$, $u_4u_5u_5u_7$, $u_7u_8u_9u_{10}$, and $u_{10}u_{11}u_{12}u_{13}u_{14}u_{1}$ in a way that ensures that the 2-vertices are properly colored.
When we reach each 3-vertex $u_k$, we will color the edge $u_ku_{k+1}$ using the same color as $u_{k-1}u_{k}$ unless $k = j$, in which case we color $u_{j}u_{j+1}$ the opposite color as $u_{j-1}u_j$.
Color the edges $u_k\ell_k$ such that the color-blind partition on $u_k$ is equal to $p_k$.
Since the edge pairs $u_1u_2$ and $u_3u_4$, $u_4u_5$ and $u_6u_7$, $u_7u_8$ and $u_9u_{10}$ must receive opposite colors, observe that the edge $u_{10}u_{11}$ will have color $a$ using this coloring.
Also observe that the edge pair $u_{10}u_{11}$ and $u_{14}u_1$ receive the same color, so the vertex $u_1$ has color-blind partition $(3,0)$ and hence we have the desired coloring $c$.
\end{proof}

We are now prepared to define $G_\phi$.
First, create all copies $V_i$ of the variable gadget $V$ for all variables $x_i$.
Then create all copies $L_j$ of the clause gadget $L$ for all clauses $C_j$.
Finally, consider each variable $\hat{x}_{i_{j,k}}$ in each clause $C_j$.
If $\hat{x}_{i_{j,k}} = x_{i_{j,k}}$, then identify the vertex $v_{6j+3}^{i_{j,k}}$ with $t_k^j$, and identify $r_{12j+5}$ and $r_{12j+6}$ with the 2-vertices adjacent to $t_k^j$, and $p_{6j+3}^{i_{j,k}}$ with the leaf $s_k^j$.
If $\hat{x}_{i_{j,k}} = \neg x_{i_{j,k}}$, then identify the vertex $v_{6j+4}^{i_{j,k}}$ with $t_k^j$, and identify $r_{12j+7}$ and $r_{12j+8}$ with the 2-vertices adjacent to $t_k^j$, and $p_{6j+4}^{i_{j,k}}$ with the leaf $s_k^j$.

Let $c$ be a 2-edge-coloring of $G_\phi$ and define the variable assignment $x_i = \begin{cases}\text{true} & c^*(v_1^i) = (2,1)\\\text{false} & c^*(v_1^i) = (3,0)\end{cases}$.
Observe that if $c^*$ is a proper vertex coloring, then $c^*(v_{6j+3}^i)= c^*(v_{1}^i)$ and $c^*(v_{6j+4}^i) \neq c^*(v_{1}^i)$ for each variable gadget $V_i$ and each clause gadget $L_j$.
Then, since $c^*$ is a proper vertex coloring, Claim~\ref{claim:clausegadgetonetrue} implies that one of the vertices $u_4$, $u_7$, $u_{10}$ in each clause gadget $L_j$ has color-blind partition $(2,1)$ and therefore the clause is satisfied by the variable assignment.
Therefore, if $\dal(G_\phi) = 2$, then $\phi$ is satisfiable.

In order to demonstrate that every satisfiable assignment corresponds to a color-blind 2-edge-coloring of $G_\phi$, we use the following claim.

\begin{claim}\label{claim:variableextension}
Let $V_i$ be a variable gadget and fix $j \in \{1,\dots,m\}$ and $t \in \{3,4\}$.
Let $D$ be the subgraph induced by the vertices $p_{6j+1},p_{6j+2},\dots,p_{6j+7}$, $v_{6j+3},v_{6j+4}$, and their neighbors.
Let $c$ be an assignment of the colors $\{1,2\}$ to the edges incident to $p_{6j+1}$ and $v_{6j+t}$ such that $c^*(p_{6j+1}) \neq c^*(v_{6j+3})$ when $t = 3$ and  $c^*(p_{6j+1}) = c^*(v_{6j+4})$ when $t = 4$.
There exists a 2-edge-coloring $c'$ of the remaining edges such that $(c\cup c')^*$ is a proper vertex coloring of $D$.
\end{claim}

Claim~\ref{claim:variableextension} follows by exhaustive enumeration of the possible colorings of the graph $D$, so the proof is omitted\footnote{The algorithm for enumerating all colorings is available as a Sage worksheet at \url{http://orion.math.iastate.edu/dstolee/r/cbindex.htm}}.

Let $x_1,\dots,x_n$ be a variable assignment such that $\phi(x_1,\dots,x_n)$ is true.
For each clause $C_j$, there is at least one variable $\hat{x}_{i_{j,k}}$ that is true, so by Claim~\ref{claim:clausegadgetsatisfy} there exists a 2-edge-coloring $c_j$ of $K_j$ where $c_j^*$ is a proper vertex coloring and the color-blind partitions of $u_4, u_7$ and $u_{10}$ correspond to the truth values of $\hat{x}_{i_{j,1}}$, $\hat{x}_{i_{j,2}}$, and $\hat{x}_{i_{j,3}}$, respectively.
Fix a 2-edge-coloring of each vertex $v_1^{i}$ such that the color-blind partition at $v_1^{i}$ corresponds to the truth value of $x_i$.
Finally, by Claim~\ref{claim:variableextension} these 2-edge-colorings of the vertices $v_1^1,\dots,v_1^n$ and clause gadgets $L_1,\dots,L_m$ extend to a 2-edge-coloring $c$ of $G_\phi$ where $c^*$ is a proper vertex coloring.

Thus, determining if $\dal(G_\phi) \leq 2$ is NP-hard.

We complete our investigation by demonstrating that $\dal(G_\phi) \leq 3$ always.
To generate a 3-edge-coloring of $G_\phi$, fix a variable assignment $x_1,\dots,x_n$.
If a clause $C_j$ is satisfied by this variable assignment, then use Claim~\ref{claim:clausegadgetsatisfy} to find a 2-edge-coloring on the clause gadget $L_j$.
If a clause $C_j$ is not satisfied by this variable assignment, then assign color 1 to the edge set
\[
	\{ z_1z_2, z_2z_3, u_1z_1, u_2u_3, u_3u_4, u_4\ell_4, u_4u_5, u_5u_6, u_9u_{10}, u_{10}\ell_{10}, u_{10}u_{11}, u_{11}u_{12} \},
\]
assign color 2 to the edge set 
\[
	\{ z_1z_3, u_1u_2, u_6u_7, u_7\ell_7, u_7u_8, u_8u_9, u_{12}u_{13}, u_{13}u_{14} \},
\]
and finally assign color 3 to the edge $u_14u_1$.
Observe that this coloring is color-blind distinguishing on $L_j$ with $c^*(u_4) = c^*(u_7) = c^*(u_{10}) = (3,0)$.
Using Claim~\ref{claim:variableextension}, this coloring extends to the variable gadgets and hence there is a color-blind distinguishing 3-edge-coloring of $G_\phi$.
\end{proof}

In the next sections, we explore determining the color-blind index of graphs using properties that avoid the constructions in the above reduction from 3-SAT.

\setcounter{theorem}{3}
\section{Regular Bipartite Graphs and 2-Colorable Hypergraphs}\label{propB}

In Section~\ref{sec:hardness}, we demonstrated that it is $\NP$-complete to determine if $\dal(G) = 2$, even when promised that $\dal(G) \in \{2,3\}$.
One particular instance of this situation is in the case of regular bipartite graphs, as Kalinowski, Pil\'sniak, Przyby\l{o}, Wo\'zniak~\cite{kalinowski2013can} determined an upper bound on the color-blind index.

\begin{theorem}[Kalinowski, Pil\'sniak, Przyby\l{o}, Wo\'zniak \upshape{\cite{kalinowski2013can}}]
If $G$ is a $k$-regular bipartite graph with $k \geq 2$, then $\dal(G) \leq 3$.
\end{theorem}

We demonstrate that when $G$ is a $k$-regular bipartite graph, $\dal(G) = 2$ if and only if at least one of two corresponding $k$-regular, $k$-uniform hypergraphs is 2-colorable.
Erd\H{o}s and Lov\'asz~\cite{erdos1975problems} implicitly proved that $k$-regular $k$-uniform hypergraphs are $2$-colorable for all $k \geq 9$ in the first use of the Lov\'asz Local Lemma.
Several results~\cite{alon1988every,bollobas1985list,branko2009configurations,vishwanathan20032} proved different cases for $k < 9$ and also demonstrated that some 3-regular 3-uniform hypergraphs are not 2-colorable, such as the Fano plane.
Thomassen~\cite{thomassen1992even} implicitly proved the general case, and Henning and Yeo~\cite{henning20132} proved it explicitly.

\begin{theorem}[Thomassen~\cite{thomassen1992even}, Henning and Yeo~\cite{henning20132}]
Let $k \geq 4$. If $\mathcal{H}$ is a $k$-regular $k$-uniform hypergraph, then $\mathcal{H}$ is 2-colorable.
\end{theorem}

McCuaig~\cite{mccuaig2004polya} has a characterization of $3$-regular, $3$-uniform, 2-colorable hypergraphs when the 2-vertex-coloring is forced to be \emph{balanced}.
A general characterization is not known for $3$-regular, $3$-uniform, 2-colorable hypergraphs.

If $\mathcal{H}$ is a $k$-uniform hypergraph with vertex set $V(\mathcal{H})$ and edge set $E(\mathcal{H})$, the \emph{vertex-edge incidence graph} of $\mathcal{H}$ is the bipartite graph $G$ with vertex set $V(\mathcal{H})\cup E(\mathcal{H})$ where a vertex $v \in V(\mathcal{H})$ and edge $e \in E(\mathcal{H})$ are incident in $G$ if and only if $v \in e$.
Note that since $\mathcal{H}$ is $k$-uniform, all of the vertices in the $E(\mathcal{H})$ part of $G$ have degree $k$; $G$ is $k$-regular if and only if $\mathcal{H}$ is also $k$-regular and $k$-uniform.

\begin{prop}
Let $\mathcal{H}$ be a $k$-uniform hypergraph and $G$ its vertex-edge incidence graph.
If $\mathcal{H}$ is 2-colorable, then $\dal(G)\leq 2$.
\end{prop}

\begin{proof}
Let $V = V(\mathcal{H})$ and $E = E(\mathcal{H})$, and let $G$ be the bipartite vertex-edge incidence graph with vertex set $V \cup E$.
Let $c : V \to \{1,2\}$ be a proper 2-vertex-coloring of $\mathcal{H}$.
For each $v \in V$, and edge $e \in E$ where $v \in e$, let the edge $ve$ of $G$ be colored $c(ve) = c(v)$.
Let $c^*$ be the color-blind partition on the vertices of $G$ induced by the coloring on the edges of $G$.
The color-blind partition at every vertex $v \in V$ is $c^*(v) = (d_{\mathcal{H}}(v),0)$.
Since $c$ is a proper 2-vertex-coloring of $\mathcal{H}$, the color-blind partition at every edge $e \in E$ is $c^*(e) = (k-i,i)$, where $1 \leq i \leq \lfloor k/2\rfloor$.
Therefore, $c^*$ is a proper vertex coloring of $G$ and $\dal(G) \leq 2$.
\end{proof}

If $G = (X \cup Y, E)$ is a $k$-regular bipartite graph, then there are two (possibly isomorphic) $k$-uniform hypergraphs $\mathcal{H}_X, \mathcal{H}_Y$, defined by $V(\mathcal{H}_X) = X$ and $E(\mathcal{H}_X) = \{ N_G(y) : y \in Y\}$,  $V(\mathcal{H}_Y) = Y$ and $E(\mathcal{H}_Y) = \{ N_G(x) : x \in X\}$.
When $k \geq 4$ and $G$ is a $k$-regular bipartite graph, then both $\mathcal{H}_X$ and $\mathcal{H}_Y$ are 2-colorable by the theorem of Henning and Yeo~\cite{henning20132}.

\begin{prop}
If $G = (X\cup Y, E)$ is a connected 3-regular bipartite graph with $\dal(G) \leq 2$, then at least one of the 3-regular, 3-uniform hypergraphs $\mathcal{H}_X$ or $\mathcal{H}_Y$ is 2-colorable.
\end{prop}

\begin{proof}
Let $c : E(G) \to \{1,2\}$ be a 2-edge-coloring of $G$ such that $c^* : V(G) \to \{ (3,0), (2,1) \}$ is a proper vertex coloring of $G$.
Then, exactly one of $X$ or $Y$ has all vertices colored with $(3,0)$ and the other is colored with $(2,1)$, since $G$ is connected.
Thus, at least one of $\mathcal{H}_X$ or $\mathcal{H}_Y$ has a 2-vertex-coloring where $c(v)$ is the unique color on the edges of $G$ incident to $v$, and this coloring is proper since each edge is incident to two vertices with one color and one vertex with the other.
\end{proof}

Since regular bipartite graphs are well understood, but the existence of a color-blind coloring in a general cubic graph is not well understood, our next section investigates cubic graphs that are as far from being bipartite as possible.

\section{Cubic Graphs with Many 3-cycles}\label{sec:3-cycles}

In this section, we prove Theorem~\ref{thm:maindiamonds} concerning 3-regular graphs where every vertex is in at least one 3-cycle.
We first demonstrate the case where $G$ has no color-blind coloring.

\begin{lemma}
If $G$ is an odd cycle of diamonds, then $\dal(G) = \infty$.
\end{lemma}

\begin{proof}
Suppose for the sake of contradiction that $c$ is a color-blind distinguishing $k$-edge-coloring of $G$ for some $k$ and note that $c^*$ takes the colors $(3,0,0), (2,1,0)$, and $(1,1,1)$.
For every diamond $xyzw$ where $xyz$ and $yzw$ are 3-cycles, observe that $c^*(x) = c^*(w)$.
Thus, for every diamond, we can associate the the diamond with the $c^*$-color of the endpoints.
Since $c^*$ is proper, adjacent diamonds must receive distinct colors.
Since an odd cycle is not 2-colorable, there must be a diamond $xyzw$ with endpoints colored $(3,0,0)$.
Then the edges $xy$ and $xz$ and the edges $wy$ and $wz$ receive the same colors under $c$.
So regardless of $c(yz)$, we must have $c^*(y) = c^*(z)$.  Hence $c^*$ is not proper.
\end{proof}

We prove Theorem~\ref{thm:maindiamonds} by using a strengthened induction, presented in Theorem~\ref{thm:diamonds}.
A \emph{$\{1,3\}$-regular graph} is a graph where every vertex has degree 1 or 3.

\begin{theorem} \label{thm:diamonds}
Let $G$ be a connected $\{1,3\}$-regular graph where every 1-vertex is adjacent to a 3-vertex and every 3-vertex is in at least one 3-cycle.
There exists a color-blind distinguishing 3-edge-coloring of $G$ if and only if $G$ is not an odd cycle of diamonds.
When a color-blind distinguishing 3-edge-coloring exists, if $v$ is a 3-vertex adjacent to a 1-vertex, then there are two color-blind distinguishing 3-edge-colorings $c_1, c_2$ where $c_1^*(v) \neq c_2^*(v)$.
\end{theorem}

\begin{proof}
Among examples of graphs that satisfy the hypothesis but do not have color-blind distinguishing 3-edge-colorings, select $G$ to minimize $n(G) + e(G)$.
We shall prove that $G$ is either a subgraph of a small list of graphs that contain color-blind distinguishing 3-edge-colorings, or contains one of a small list of \emph{reducible configurations}.

Figure~\ref{fig:base} lists four graphs and demonstrates color-blind distinguishing 3-edge-colorings that satisfy the theorem statement.
Therefore, $G$ is not among this list.

\begin{figure}[htp]
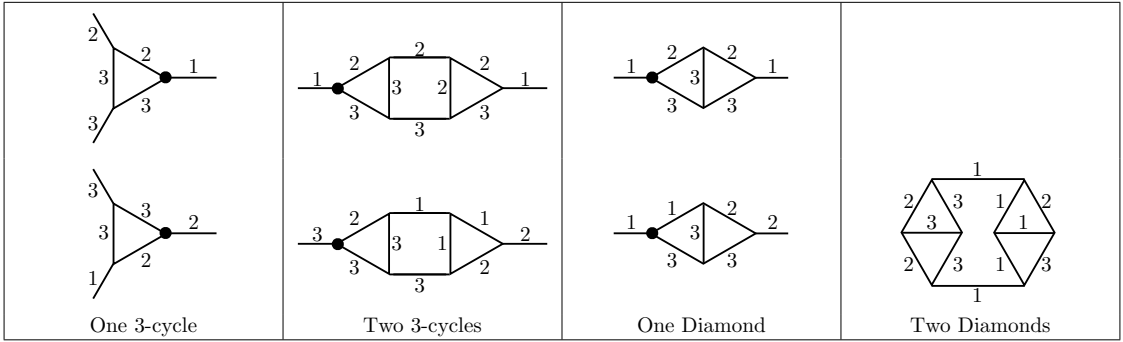

\centering
\scalebox{0.8}{
\begin{tabular}[h]{|c|c|c|c|}
\hline
\begin{lpic}[]{"figures/base-t1"(42mm,)}
\lbl[b]{50,22;1}
\lbl[bl]{35,25.5;2}
\lbl[tl]{35,15;3}
\lbl[r]{26,21;3}
\lbl[tr]{23,35;2}
\lbl[br]{23,6;3}
\end{lpic}
&
\begin{lpic}[]{"figures/base-t2a"(42mm,)}
\lbl[b]{6,18.5;1}
\lbl[br]{18,23;2}
\lbl[tr]{18,13;3}
\lbl[l]{27,18;3}
\lbl[b]{35,27;2}
\lbl[t]{35,8;3}
\lbl[r]{43,18;2}
\lbl[bl]{52,23;2}
\lbl[tl]{52,13;3}
\lbl[b]{65,18.5;1}
\end{lpic}
&
\begin{lpic}[]{"figures/base-b1t0"(42mm,)}
\lbl[b]{16,22;1}
\lbl[br]{29,26;2}
\lbl[tr]{29,15;3}
\lbl[r]{35.5,21;3}
\lbl[tl]{43,15;3}
\lbl[bl]{43,26;2}
\lbl[b]{56,22;1}
\end{lpic}
&
\\
\begin{lpic}[]{"figures/base-t1"(42mm,)}
\lbl[b]{50,22;2}
\lbl[bl]{35,25.5;3}
\lbl[tl]{35,15;2}
\lbl[r]{26,21;3}
\lbl[tr]{23,35;3}
\lbl[br]{23,6;1}
\end{lpic}
&
\begin{lpic}[]{"figures/base-t2a"(42mm,)}
\lbl[b]{6,18.5;3}
\lbl[br]{18,23;2}
\lbl[tr]{18,13;3}
\lbl[l]{27,18;3}
\lbl[b]{35,27;1}
\lbl[t]{35,8;3}
\lbl[r]{43,18;1}
\lbl[bl]{52,23;1}
\lbl[tl]{52,13;2}
\lbl[b]{65,18.5;2}
\end{lpic}
&
\begin{lpic}[]{"figures/base-b1t0"(42mm,)}
\lbl[b]{16,22;1}
\lbl[br]{29,26;1}
\lbl[tr]{29,15;3}
\lbl[r]{35.5,21;3}
\lbl[tl]{43,15;3}
\lbl[bl]{43,26;2}
\lbl[b]{56,22;2}
\end{lpic}
&
\begin{lpic}[]{"figures/base-b2a"(42mm,)}
\lbl[b]{35,37;1}
\lbl[br]{17,28;2}
\lbl[tr]{17,14;2}
\lbl[b]{22,22;3}
\lbl[bl]{28,28;3}
\lbl[tl]{28,14;3}
\lbl[t]{35,5;1}
\lbl[br]{43,28;1}
\lbl[tr]{43,14;1}
\lbl[b]{48,22;1}
\lbl[bl]{53,28;2}
\lbl[tl]{53,14;3}
\end{lpic}
\\
One 3-cycle & Two 3-cycles  & One Diamond & Two Diamonds \\
\hline
\end{tabular}
}
\caption{\label{fig:base}Base case graphs and their color-blind distinguishing 3-edge-coloring(s) where the 3-vertex $v$ is highlighted.}
\end{figure}

\begin{claim}\label{lma:cutedge}
$G$ does not contain a cut-edge $uv$ where $d(u) = d(v) = 3$.
\end{claim}

\begin{proof}
Suppose $uv$ is a cut-edge with $d(u) = d(v) = 3$.
Let $G_1$ and $G_2$ be the components of $G - uv$ where $u \in V(G_1)$ and $v \in V(G_2)$, and let $G_i' = G_i + uv$ for each $i \in \{1,2\}$.
Observe that $n(G_i') + e(G_i') < n(G) + e(G)$.
Also, neither is an odd cycle of diamonds, as they have vertices of degree 1.
Therefore, there are color-blind distinguishing 3-edge-colorings $c_1 : E(G_1') \to \{1,2,3\}$
and $c_2 : E(G_2') \to \{a,b,c\}$ such that $c_2^*(v) \neq c_1^*(u)$.
The color set $\{a,b,c\}$ can be permuted to $\{1,2,3\}$ such that $c_2(uv)$ is mapped to $c_1(uv)$.
Under this permutation, $c_1$ and $c_2$ combine to form a color-blind distinguishing 3-edge-coloring of $G$.
\end{proof}

Further note that every color-blind distinguishing 3-edge-coloring of $G_1'$ extends to a color-blind distinguishing 3-edge-coloring of $G$, and by symmetry every color-blind distinguishing 3-edge-coloring of $G_2'$ extends to a color-blind distinguishing 3-edge-coloring of $G$.
Thus, for any vertex $x$ of degree 3 adjacent to a vertex of degree 1, $x$ is also a vertex of degree 3 adjacent to a vertex of degree 1 in some $G_i'$ and hence has distinct color-blind partitions for two colorings in that $G_i'$.
These colorings both extend to $G$, so the two distinct color-blind partitions on $x$ also appear in color-blind distinguishing 3-edge-colorings of $G$.

\begin{defn}[Reducible Configurations]
Let $H$ be a $\{1,3\}$-regular graph and $D \subset V(H)$ such that for every $v \in D$, there is at most one vertex $u = u(v) \in N_H(v) \setminus D$.
Let $H_D$ be the subgraph of $H$ induced by $D$, let $S$ be the set of neighbors of $D$ that are not in $D$.
Let $M$ be a matching that saturates $S$, using edges in the edge cut $[D,S]$ or using pairs from $S$.

Let $c : M \to \{1,2,3\}$ and $c^* : S \to \{ (3,0,0), (2,1,0), (1,1,1)\}$ be assignments such that $c^*(u) \neq c^*(v)$ for all edges $uv \in M$.
For an edge $xy \in [D,S]$ where $x \in D$ and $y \in S$, define $c(xy)$ to be $c(yz)$ where $yz$ is the edge of $M$ covering $y$.
Such a pair $(c,c^*)$ is a \emph{potential pair}.

The triple $(H,D,M)$ is a \emph{reducible configuration} if $E(H_D)\neq \varnothing$, and for every potential pair $(c,c^*)$, there exists an extension of $c$ to include $E(H_D)$ where the color-blind partitions for vertices in $D$ create an extension of $c^*$ to $D$ that is a proper vertex coloring of $D \cup S$.
We say that a graph $G$ \emph{contains} a reducible configuration $(H,D,M)$ if it contains $H$ as a subgraph, and the corresponding vertices of $S$ in that subgraph have degree 3 in $G$.
\end{defn}

Figure~\ref{fig:reducibleconfigs} contains a list of four reducible configurations.
Some of these are checkable by hand, while others were verified to be reducible using a computer\footnote{The algorithm for checking reducibility is available as a Sage worksheet at \url{http://orion.math.iastate.edu/dstolee/r/cbindex.htm}}.
If $(H,D,M)$ is a reducible configuration and $H \subseteq G$, we use $G - D + M$ to denote the \emph{reduced graph} given by deleting the edges with at least one endpoint in $D$, adding the edges in $M$, and removing any isolated vertices.
Observe that for every reducible configuration in Figure~\ref{fig:reducibleconfigs}, every vertex $x$ in $G$ (not in $D$) that has degree 3 and is adjacent to a vertex of degree 1 remains a vertex of this type  in the reduced graph $G - D + M$.
Therefore, the two colorings that provide distinct color-blind partitions for $x$ in $G - D + M$ each extend to a color-blind distinguishing 3-edge-coloring of $G$.

\begin{figure}[htp]
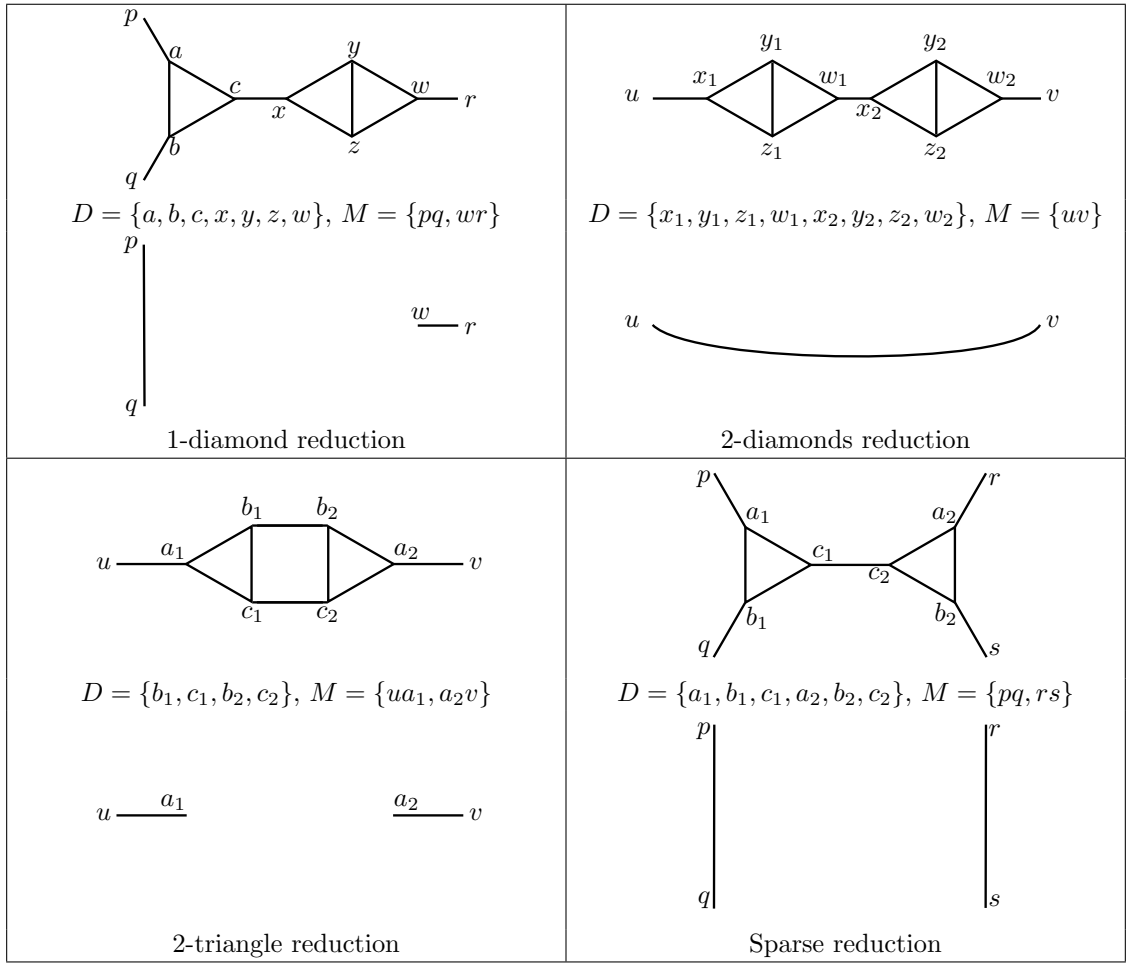

\centering
\scalebox{1}{
\begin{tabular}[htp]{|m{70mm}|m{70mm}|}
\hline
\begin{lpic}[]{"figures/reducible-1buoy-orig"(70mm,)}
\lbl[t]{58,19;$x$}
\lbl[b]{75,30;$y$}
\lbl[t]{75,11;$z$}
\lbl[b]{90,22;$w$}
\lbl[bl]{33,30;$a$}
\lbl[tl]{33,12;$b$}
\lbl[b]{48,22;$c$}
\lbl[r]{26,39;$p$}
\lbl[r]{26,2;$q$}
\lbl[l]{100,20;$r$}
\end{lpic}
&
\begin{lpic}{"figures/reducible-2buoys-orig"(70mm,)}
\lbl[r]{13,10;$u$}
\lbl[l]{105,10;$v$}
\lbl[b]{28,12;$x_1$}
\lbl[b]{43,20;$y_1$}
\lbl[t]{43,-1;$z_1$}
\lbl[b]{57,12;$w_1$}
\lbl[t]{65,8;$x_2$}
\lbl[b]{80,20;$y_2$}
\lbl[t]{80,-1;$z_2$}
\lbl[b]{95,12;$w_2$}
\end{lpic}
\\
\multicolumn{1}{|c|}{$D = \{a,b,c,x,y,z,w\}$, $M = \{ pq, wr \}$ } &
\multicolumn{1}{c|}{$D = \{x_1, y_1, z_1, w_1, x_2, y_2, z_2, w_2\}$, $M = \{ uv \}$}
\\
\begin{lpic}[]{"figures/reducible-1buoy-reduced"(70mm,)}
\lbl[b]{90,22;$w$}
\lbl[r]{26,39;$p$}
\lbl[r]{26,2;$q$}
\lbl[l]{100,20;$r$}
\end{lpic}
&
\begin{lpic}{"figures/reducible-2buoys-reduced"(70mm,)}
\lbl[r]{13,10;$u$}
\lbl[l]{105,10;$v$}
\end{lpic}
\\
\multicolumn{1}{|c|}{1-diamond reduction} &
\multicolumn{1}{c|}{2-diamonds reduction}
\\\hline
\begin{lpic}[]{"figures/reducible-2triangles-orig"(70mm,)}
\lbl[r]{20,18;$u$}
\lbl[br]{37,19;$a_1$}
\lbl[b]{52,28;$b_1$}
\lbl[t]{52,8;$c_1$}
\lbl[bl]{84,19;$a_2$}
\lbl[b]{69,28;$b_2$}
\lbl[t]{69,8;$c_2$}
\lbl[l]{101,18;$v$}
\end{lpic}
&
\begin{lpic}[]{"figures/reducible-final-orig"(70mm,)}
\lbl[bl]{37,33;$a_1$}
\lbl[tl]{37,14;$b_1$}
\lbl[bl]{52,24.5;$c_1$}
\lbl[tr]{70,23;$c_2$}
\lbl[br]{85,33;$a_2$}
\lbl[tr]{85,15;$b_2$}
\lbl[r]{29,43;$p$}
\lbl[r]{29,4.5;$q$}
\lbl[l]{92,43;$r$}
\lbl[l]{92,4.5;$s$}
\end{lpic}
\\
\multicolumn{1}{|c|}{$D = \{ b_1, c_1, b_2, c_2 \}$, $M = \{ ua_1, a_2v \}$} &
\multicolumn{1}{c|}{$D = \{ a_1, b_1, c_1, a_2, b_2, c_2\}$, $M = \{ pq, rs \}$}\\
\begin{lpic}[]{"figures/reducible-2triangles-reduced"(70mm,)}
\lbl[r]{20,18;$u$}
\lbl[br]{37,19;$a_1$}
\lbl[bl]{84,19;$a_2$}
\lbl[l]{101,18;$v$}
\end{lpic}
&
\begin{lpic}[]{"figures/reducible-final-reduced"(70mm,)}
\lbl[r]{29,43;$p$}
\lbl[r]{29,4.5;$q$}
\lbl[l]{92,43;$r$}
\lbl[l]{92,4.5;$s$}
\end{lpic}
\\
\multicolumn{1}{|c|}{2-triangle reduction} &
\multicolumn{1}{c|}{Sparse reduction}
\\\hline
\end{tabular}
}
\caption{\label{fig:reducibleconfigs}The Reducible Configurations and their Reductions.}
\end{figure}

\begin{claim}
Let $H$ be the graph in the 1-diamond reduction, with $D = \{a,b,c,	\}$ and $M = \{ pq, wr\}$.
$G$ does not contain the reducible configuration $(H,D,M)$.
\end{claim}

\begin{proof}
Suppose $G$ contains $H$ as a subgraph.
Let $G' = G - D + M$, and observe that $n(G') + e(G') < n(G) + e(G)$.
Also, by Claim \ref{lma:cutedge} the edge $cx$ is not a cut-edge of $G$, $G'$ is connected and is not an odd cycle of diamonds.
Therefore, there exists a color-blind distinguishing 3-edge-coloring $c : E(G') \to \{1,2,3\}$ where $c^*$ is a proper vertex coloring on $G'$ and hence a proper vertex coloring on $M$.
By the definition of reducible configuration, this coloring $c$ extends to a color-blind distinguishing 3-edge-coloring of $G$, a contradiction.
\end{proof}

\begin{claim}
Let $H$ be the graph in the 2-diamond reduction, with $D = \{x_1,y_1,z_1,w_1,x_2,y_2,z_2,w_2\}$ and $M = \{ uv \}$.
$G$ does not contain the reducible configuration $(H,D,M)$.
\end{claim}

\begin{proof}
Suppose $G$ contains $H$ as a subgraph.
Let $G' = G - D + M$, and observe that $n(G') + e(G') < n(G) + e(G)$.
Also, since $G$ is not an odd cycle of diamonds, $G'$ is not an odd cycle of diamonds.
Therefore, there exists a color-blind distinguishing 3-edge-coloring $c : E(G') \to \{1,2,3\}$ where $c^*$ is a proper vertex coloring on $G'$ and hence a proper vertex coloring on $M$.
By the definition of reducible configuration, this coloring $c$ extends to a color-blind distinguishing 3-edge-coloring of $G$, a contradiction.
\end{proof}


\begin{claim}
Let $H$ be the graph in the 2-triangle reduction, with $D = \{b_1,c_1,b_2,c_2\}$ and $M = \{ ua_1, a_2v \}$.
$G$ does not contain the reducible configuration $(H,D,M)$.
\end{claim}

\begin{proof}
Suppose $G$ contains $H$ as a proper subgraph.
Since $G$ is connected, at least one of $u$ and $v$ is a 3-vertex; without loss of generality $u$ is a 3-vertex.
The edge $ua_1$ is not a cut-edge, by \ref{lma:cutedge}, so $v$ is also a 3-vertex.
Let $G' = G - D + M$, and observe that $n(G') + e(G') < n(G) + e(G)$.
Observe that $G'$ is connected, $\{1,3\}$-regular and not an odd cycle of diamonds, and that every 1-vertex is adjacent to a 3-vertex.
Therefore, there exists a color-blind distinguishing 3-edge-coloring $c : E(G') \to \{1,2,3\}$ where $c^*$ is a proper vertex coloring on $G'$ and hence a proper vertex coloring on $M$.
By the definition of reducible configuration, this coloring $c$ extends to a color-blind distinguishing 3-edge-coloring of $G$, a contradiction.
\end{proof}

\begin{claim}
Let $H$ be the graph in the sparse reduction, with $D = \{a_1,b_1,c_1,a_2,b_2,c_2\}$ and $M = \{ pq, rs \}$.
$G$ does not contain the reducible configuration $(H,D,M)$.
\end{claim}

\begin{proof}
Suppose $G$ contains $H$ as a subgraph.
Let $G' = G - D + M$, and observe that $n(G') + e(G') < n(G) + e(G)$.
Since $G$ does not contain the 2-triangle reduction, $pq$ and $rs$ are not edges of $G$ and hence $G'$ is a $\{1,3\}$-regular graph.
Also, since the edge $c_1c_2$ is not a cut-edge and since $G$ does not contain the 1-diamond reduction, $G'$ is connected and is not an odd cycle of diamonds.
Therefore, there exists a color-blind distinguishing 3-edge-coloring $c : E(G') \to \{1,2,3\}$ where $c^*$ is a proper vertex coloring on $G'$ and hence a proper vertex coloring on $M$.
By the definition of reducible configuration, this coloring $c$ extends to a color-blind distinguishing 3-edge-coloring of $G$, a contradiction.
\end{proof}

We complete our proof by demonstrating that $G$ contains one of the reducible configurations.

Suppose that $G$ contains a diamond $xyzw$ where $xyzw$ is a 4-cycle and $yz$ is an edge.
Since $G$ is not a single diamond, without loss of generality we have that the vertex adjacent to $x$, say $u$, is in a 3-cycle or a diamond.
Therefore, $G$ is isomorphic to or contains either the 1-diamond reduction or the 2-diamonds reduction.
We may now assume that $G$ does not contain any diamond.

Let $abc$ be a 3-cycle in $G$.
Since $G$ has more than one 3-cycle, at least one vertex is adjacent to a vertex in another 3-cycle.
If two vertices in $\{a,b,c\}$ are adjacent to the same 3-cycle, then $G$ is isomorphic to or contains the 2-triangle reduction.
Therefore, we may assume that every pair of adjacent 3-cycles have exactly one edge between them.
However, a pair of adjacent 3-cycles and their neighboring vertices form a sparse reduction as a subgraph of $G$.

Therefore, the minimal counterexample $G$ does not exist and the theorem holds.
\end{proof}

\begin{remark}
The use of reducible configurations demonstrates a polynomial-time algorithm for finding a color-blind distinguishing 3-edge-coloring of a cubic graph where every vertex is in exactly one 3-cycle.
The algorithm works recursively, with base cases among the list of two diamonds or two 3-cycles where the two color-blind distinguishing 3-edge-colorings $c_1$ and $c_2$ can be produced in constant time.
The algorithm first searches for a cut-edge $uv$ with $d(u) = d(v) = 3$ and if one exists creates the graphs $G_1'$ and $G_2'$ as in Claim~\ref{lma:cutedge}; recursion on these graphs produces colorings that can be combined to form a coloring of $G$.
If no such cut-edge is found, then the algorithm searches for one of the reducible configurations, one of which will exist.
By performing the reduction from $G$ to $G - D + M$ as specified by the reducible configuration, the algorithm can recursively find a coloring on $G - D + M$ and in constant time produce an extension to $G$.
\end{remark}

\section*{Acknowledgments}

This collaboration began as part of the 2014 Rocky Mountain--Great Plains Graduate Research Workshop in Combinatorics, supported in part by NSF Grant \#1427526.
The authors thank Jessica DeSilva and Michael Tait for early discussions about this problem.

\bibliographystyle{abbrv}
\bibliography{colorblind}

\clearpage\small
\appendix

\section{Trees and Cacti}\label{sec:cacti}

In this section, we detail how to color trees and cacti using very few colors.
The proof is very technical, so we leave it in this appendix.

Recall that a connected graph is a \emph{cactus} if every vertex is contained in at most one cycle.
A special case is a \emph{tree}, which contains no cycles.
Trees have color-blind index at most two using a simple coloring.

\begin{prop}
If $T$ is a tree on  at least three vertices, then $\dal(T) \leq 2$.
\end{prop}

\begin{proof}
Fix a vertex $v_0 \in V(T)$.
For $i \geq 0$, define the \emph{$i^{th}$ neighborhood of $v_0$}, denoted $N_i(v_0)$, to be the set of vertices in $T$ at distance $i$ from $v_0$.
Observe that each $N_i(v_0)$ is an independent set.
Let $c$ assign the color $1$ to the edges with endpoints in $N_i(v_0)$ and $N_{i+1}(v_0)$ if and only if $i \equiv 0 \pmod 4$ or $i \equiv 3 \pmod 4$, and assign the color $2$ to all other edges.  Since $T$ is not $K_2$, each 1-vertex has a color-blind partition distinct from its neighbor.  Let $v$ be a vertex not a 1-vertex.
When $i \geq 0$ is even and $v\in N_i(v_0)$, $c^*(v)=(d(v), 0)$.
When $i \geq 1$ is odd and $v\in N_i(v_0)$, $c^*(v)=(d(v)-1,1)$.
Therefore, $c^*$ is a proper vertex coloring of $V(T)$, and $c$ is a color-blind distinguishing 2-edge-coloring.
\end{proof}

A connected graph $G$ is a \emph{cactus} if every edge is contained in at most one cycle.
Equivalently, a cactus is a connected graph where every block is a cycle or a copy of $K_2$, where a \emph{block} is a maximally 2-connected subgraph or a cut-edge.
The \emph{block-cutpoint tree} of $G$ is the tree whose vertices are the blocks and cut-vertices of $G$, and an edge exists between a block  $B$ and a cut-vertex $v$ if $v$ is contained in $B$.
We will build a color-blind distinguishing coloring by iteratively coloring blocks starting at some cut-vertex.
If a block contains a cut-vertex whose incident edges are colored, then we will extend the coloring to that block, which then may determine the colors incident to some other cut-vertices.
Since the block-cutpoint tree has no cycles, we can continue this procedure and every block will eventually be colored by extending the coloring around a single vertex.

If $B$ is a block of $G$, the \emph{extended block} $\overline{B}$ is the subgraph of $G$ given by the edges with at least one endpoint in $B$.
Observe that in a cactus, the extended block of a cut-edge is a double-star and the extended block of a cycle is a cycle possibly with additional pendant edges -- which we call a \emph{hairy cycle}.

Determining the color-blind index of a cactus is nontrivial even when there is exactly one block: a cycle.
Kalinowski, Pli\'snia, Przyby\l{}o, and Wo\'zniak determined the color-blind index of cycles.

\begin{theorem}[Kalinowski, Pli\'snia, Przyby\l{}o, Wo\'zniak \upshape{\cite{kalinowski2013can}}] \label{thm:cycles}
For $n \geq 3$,
\[\dal(C_n)=\begin{cases} 2 & \text{ if } n\equiv 0\pmod{4}\\ 3 & \text{ if } n\equiv 2\pmod{4}\\\infty & \text{ if } n \equiv 1 \pmod{2}.\end{cases}\]
\end{theorem}

We mostly determined the color-blind index of cacti, as summarized in the following two theorems.

\begin{theorem}\label{thm:cacti3}
If $G$ is a cactus that is not an odd cycle or a single edge, then $\dal G \leq 3$.
\end{theorem}

\begin{theorem}\label{thm:cacti2}
Let $G$ be a cactus with at least two blocks and the set of vertices of degree two form an independent set.
Then $\dal G \leq 2$ if and only if $G$ contains no 3-uniform odd cycle.
\end{theorem}

Recall that $\dal G = 1$ if and only if no two adjacent vertices have the same degree.
We require that no two vertices of degree two are adjacent in Theorem~\ref{thm:cacti2} in order to avoid very tricky situations that occur, especially in the rigid properties of the clause gadget from Section~\ref{sec:hardness}.

To prove Theorems~\ref{thm:cacti2} and \ref{thm:cacti3}, we use a strengthened induction.
For $k \in \{2, 3\}$, we will determine how to $k$-color the edges of a cactus to produce a color-blind distinguishing coloring.
We will always use the color set $\{1,2\}$ for a 2-edge-coloring and the color set $\{1,2,3\}$ for a 3-edge-coloring.
However, we use a semi-greedy method by coloring each extended block one at a time.
Observe that two extended blocks that intersect on at least one edge share exactly the edges incident to a single vertex.
Therefore, we will consider \emph{extending} a $k$-coloring on the edges incident to a single vertex to the rest of the extended block.
This extension is made concrete in the following lemmas.
These lemmas are stated where we start with a $k$-coloring for some $k \in \{2,3\}$ and we will use two or three colors to extend the coloring, depending on how many colors are needed.
As the number of colors may change, we will prune the color-blind partitions by deleting any trailing zeroes. Thus, $(3,0,0)$ and $(3,0)$ become $(3)$ and $(2,1,0)$ becomes $(2,1)$.

As we will consider \emph{partial} colorings that do not assign color to every edge of $G$, we say a coloring $c$ is \emph{color-blind distinguishing among $S$} for a set $S \subseteq V(G)$ if all edges with at least one endpoint in $S$ are colored, and $c^*$ is a proper vertex coloring of the subgraph induced by $S$.

For a block isomorphic to $K_2$, we can extend a coloring from one vertex to the other.

\begin{lemma}\label{lma:doublestar}
Let $S$ be a double-star with center vertices $u$ and $v$, where $d(u), d(v) \geq 2$.
For $k \in \{2, 3\}$, let $c$ be a $k$-coloring of the edges incident to $u$.
There exists a $2$-coloring $c'$ of the edges incident to $v$ such that $c \cup c'$ is color-blind distinguishing among $\{u,v\}$.
\end{lemma}

\begin{proof}
If $d(u) \neq d(v)$, then let $c'(vx) = c'(uv)$ for all $x \in N(v)$.
Since no adjacent vertices in $S$ have the same degree, $c \cup c'$ is color-blind distinguishing.

If $d(u) = d(v) = d$, then consider the color-blind partition at $u$ to be $c^*(u) = (d-i,i)$ for some $i \leq \lfloor d/2\rfloor$.
If $i = 0$, then select a color $a$ from $\{1,2\} \setminus \{ c(uv)\}$ and color $c'(vx) = a$ for all $x \in N(v) \setminus \{u\}$.
In this case $c \cup c'$ is color-blind distinguishing since the color-blind partition at $v$ is $(d-1,1)$.
Otherwise, $i > 0$ and set $c'(vx) = c(uv)$ for all $x \in N(v)$.
In this case, $c\cup c'$ is color-blind distinguishing since the color-blind partition at $v$ is $(d,0)$.
\end{proof}

Blocks isomorphic to cycles are a bit more complicated.
We can start by partitioning the vertices of a cycle into parts based on the degrees of the vertices.
If the cycle is uniform, then there is only one part; otherwise there are multiple paths where all vertices have the same degree.
Between distinct parts the vertices change degree, so no conflicts can occur there.
We can extend around the cycle using extensions along the paths, but at least one path needs to extend the coloring from a partial coloring on both endpoints.

First, we consider a path of vertices of degree 3 and assume all edges incident to the first vertex are colored.
We extend the coloring using two colors to be color-blind distinguishing among the vertices on the path.

\begin{lemma}\label{lma:3uniformpath}
Let $v_1v_2\dots v_t$ be a path in $G$ where $t \geq 3$ and  $d(v_i) = 3$ for all $i \in \{1,\dots,t\}$.
For $k \in \{2, 3\}$, let $c$ be a $k$-coloring of the edges incident to $v_1$.
There exists a $2$-coloring $c'$ of the edges incident to $v_2,\dots,v_t$ such that $c \cup c'$ is color-blind distinguishing among $\{v_1,\dots,v_t\}$.
\end{lemma}

\begin{proof}
Let $u_2,\dots,u_{t-1}$ be the neighbors of $v_2,\dots,v_{t-1}$ that are not on the path.
Let $u_t$ and $v_{t+1}$ be the neighbors of $v_t$ that are not on the path.
Let $a = c(v_1v_2)$ and select $b \in \{1,2\} \setminus \{a \}$ and $d \in \{1,2\} \setminus \{b\}$; observe that $a = d$ if and only if $a\neq 3$.

If $c^*(v_1) = (3)$, then color the edges $v_iv_{i+1}$ with color $b$ for all $i \in \{2,\dots, t\}$, $v_{2i}u_{2i}$ with color $d$ for $i \in \{1,\dots,\lfloor t/2\rfloor\}$, and $v_{2i+1}u_{2i+1}$ with color $b$ for $i \in \{1,\dots, \lfloor (t-1)/2\rfloor\}$.
Observe $c^*(v_2) \in \{ (2,1), (1,1,1)\}$, $c^*(v_{i}) = (2,1)$ when $i>2$ is even  that $c^*(v_{i}) = (3)$ when $i$ is odd, so $c \cup c'$ is color-blind distinguishing among $v_1,\dots, v_t$.

If $c^*(v_1) = (2,1)$ and $a = d$, then color the edges $v_iv_{i+1}$ with color $a$ for all $i \in \{2,\dots, t\}$, $v_{2i}u_{2i}$ with color $a$ for $i \in \{1,\dots,\lfloor t/2\rfloor\}$, and $v_{2i+1}u_{2i+1}$ with color $b$ for $i \in \{1,\dots, \lfloor (t-1)/2\rfloor\}$.
Observe $c^*(v_{i}) = (3)$ when $i$ is even and $c^*(v_{i}) = (2,1)$ when $i$ is odd, so $c \cup c'$ is color-blind distinguishing among $v_1,\dots, v_t$.

If $c^*(v_1) = (2,1)$ and $a \neq d$, then color the edges $v_iv_{i+1}$ with color $b$ for all $i \in \{2,\dots, t\}$, $v_{2i}u_{2i}$ with color $d$ for $i \in \{1,\dots,\lfloor t/2\rfloor\}$, and $v_{2i+1}u_{2i+1}$ with color $b$ for $i \in \{1,\dots, \lfloor (t-1)/2\rfloor\}$.
Observe $c^*(v_2)= (1,1,1)$, $c^*(v_{i}) = (2,1)$ when $i>2$ is even and $c^*(v_{i}) = (3)$ when $i$ is odd, so $c \cup c'$ is color-blind distinguishing among $v_1,\dots, v_t$.
\end{proof}

Now, we consider a path of vertices that all have the same degree $d \geq 4$, and assume all edges incident to the first vertex are colored and the last vertex has one colored edge and extend the coloring using two colors to be color-blind distinguishing among the vertices on the path.

\begin{lemma}\label{lma:4plusuniformpath}
Let $v_1v_2\dots v_t$ be a path in $G$ where $t \geq 2$ and $d(v_i) = d \geq 4$ for all $i \in \{1,\dots,t\}$.
For $k \in \{2, 3\}$, let $c$ be a $k$-coloring of the edges indicent to $v_1$ and one edge incident to $v_t$ other than $v_{t-1}v_t$.
There exists a $2$-coloring $c'$ of the edges incident to $v_2,\dots,v_{t-1}$ such that $c \cup c'$ is a-blind distinguishing coloring among $\{v_1,\dots,v_t\}$.
\end{lemma}

\begin{proof}
Let $v_{t+1}$ be the vertex such that the colored edge incident to $v_t$ is $v_tv_{t+1}$.
For $i \in \{2,\dots, t\}$ let $u_{i,1},\dots,u_{i,d-2}$ be the neighbors of $v_i$ other than $v_{i-1}$ and $v_{i+1}$.
Color all edges $v_iv_{i+1}$ with color $1$ for $i \in \{2,\dots, t-1\}$.

If $c^*(v_1) = (d)$, then for all $i \in \{2,\dots,t-1\}$ assign colors from $\{1,2\}$ to the edges $v_iu_{i,j}$ such that $c^*(v_i) = (d-1,1)$  when $i$ is even and $c^*(v_i) = (d)$ when $i$ is odd.
Depending on the parity of $t$ and if $c(v_tv_{t+1}) \in \{1,2\}$, assign colors from $\{1,2\}$ to the edges $v_tu_{t,j}$ such that $c^*(v_t) \in \{ (d), (d-1,1), (d-2,2), (d-2,1,1)\} \setminus \{ c^*(v_{t-1})\}$.
The resulting coloring is color-blind distinguishing among $v_1,\dots,v_t$.

If $c^*(v_1) \neq (d)$, then for all $i \in \{2,\dots,t-1\}$ assign colors from $\{1,2\}$ to the edges $v_iu_{i,j}$ such that $c^*(v_i) = (d)$  when $i$ is even and $c^*(v_i) = (d-1,1)$ when $i$ is odd.
Depending on the parity of $t$ and if $c(v_tv_{t+1}) \in \{1,2\}$, assign colors from $\{1,2\}$ to the edges $v_tu_{t,j}$ such that $c^*(v_t) \in \{ (d), (d-1,1), (d-2,2), (d-2,1,1)\} \setminus \{ c^*(v_{t-1})\}$.
The resulting coloring is color-blind distinguishing among $v_1,\dots,v_t$.
\end{proof}

We now combine the above path-extension lemmas to demonstrate that we can extend a coloring to a cycle block when the edges incident to one vertex are colored.

\begin{lemma}\label{lma:hairy}
Let $C$ be a hairy cycle with no adjacent 2-vertices, and  let $v$ be a vertex on the cycle of $C$.
Let $k \in \{2, 3\}$ and let $c$ be a $k$-coloring of the edges incident to $v$.
\begin{enumerate}
\item If $C$ is a 3-uniform hairy odd cycle, then there is a 3-edge-coloring $c'$ of $C$ not incident to $v$ such that $c \cup c'$ is color-blind distinguishing among the vertices on the cycle.
\item If $C$ is not a 3-uniform hairy odd cycle, then there is a 2-edge-coloring $c'$ of the edges of $C$ not incident to $v$ such that $c \cup c'$ is color-blind distinguishing among the vertices on the cycle.
\end{enumerate}
\end{lemma}

\begin{proof}
We consider the two cases: $C$ is a 3-uniform hairy odd cycle, and $C$ is not a 3-uniform hairy odd cycle.

When $C$ is a 3-uniform hairy odd cycle, let $v_1,\dots,v_{2i+1}$ be the vertices along the cycle, where $v_1 = v$.
Every vertex $v_i$ has degree three and thus has one incident edge not on the cycle.
Let $u_i$ be the neighbor of $v_i$ not on the cycle.
Let $a= c(v_1v_2)$ and $b = c(v_{2i+1}v_1)$.

If $a \neq b$, then let $d \in \{1,2,3\}\setminus \{a,b\}$ and color the edges $v_{2j}v_{2j+1}$ with color $a$ for $j \in \{1, \dots, i-1 \}$ and the edges $v_{2j+1}v_{2j+2}$ with color $b$ for $j \in \{1,\dots, i-1 \}$.
Color the edges $v_{2j}u_{2j}$ with color $a$ for $j \in \{1,\dots,i\}$ and the edges $v_{2j+1}u_{2j+1}$ with color $d$ for $j \in \{1,\dots,i-1\}$.
Finally, color the edges $v_{2i}v_{2i+1}$ and $v_{2i+1}u_{2i+1}$ with color $b$.
Thus,  $c^*(v_{2j+1}) = (1,1,1) \neq (2,1,0) = c^*(v_{2j+2})$ for all $j \in \{1, \dots, i-1 \}$ and $c^*(v_2) = c^*(v_{2i+1}) = (3,0,0) \ne ^*(v_1)$.

If $a = b$, then let $\{ d,e\} = \{1,2,3\}\setminus \{a,b\}$ and color the edges $v_{j}v_{j+1}$ with color $d$ for $j \in \{2,\dots,2i\}$.
Color the edges $v_{2j}u_{2j}$ with color $e$ for $j \in \{1,\dots,i\}$ and the edges $v_{2j+1}u_{2j+1}$ with color $d$ for $j \in \{1,\dots,i-1\}$.
Finally, color $v_{2i+1}u_{2i+1}$ with color $e$.
Thus,  $c^*(v_{2j}) = (2,1,0) \neq c^*(v_1)$ for all $j \in \{2,\dots,i\}$, $c^*(v_{2j+1}) = (3,0,0)$ for all $j \in \{1,\dots,i-1\}$, and $c^*(v_2) = c^*(v_{2i+1}) = (1,1,1) \ne c^*(v_1)$.

When $C$ is not a 3-uniform hairy odd cycle, let $v_1,\dots,v_{n}$ be the vertices along the cycle, where $v_1 = v$.
Let $a = c(v_1v_2)$.  If $c(v_nv_1) = a$, then let $b \in \{1, 2\} \setminus \{a\}$; otherwise, let $b = c(v_nv_1)$.

If $C$ is $d$-uniform for some $d \geq 4$, then let $u_{j,1},\dots,u_{j,d-2}$ be the 1-vertices adjacent to $v_j$ for $j \in \{1,\dots,n\}$.
Starting with $f(v_1) = c^*(v_1)$, assign a proper 3-vertex-coloring $f$ to the vertices $v_1,\dots,v_n$ such that $f(v_j) \in \{ (d,0), (d-1,1), (d-2,2)\}$ and $f(v_n) \neq (d,0)$.
Color the edges $v_jv_{j+1}$ with color $a$ for all $j \in \{2,\dots,n-1\}$.
Observe that every vertex $v_j$ with $j \in \{2,\dots,n-1\}$ is adjacent to two edges with color $a$.
Color the edges $v_ju_{j,i}$ with color $a$ for $i \in \{1,\dots,k-2\}$ where $f(v_j) = (k,d-k)$, and color the other edges $v_ju_{j,i}$ with color $b$ for $i \in \{k-1,\dots,d-2\}$; hence $c^*(v_j) = f(v_j)$ for $j \in \{2,\dots,n-1\}$.
Observe that since $f(v_n) \neq (d,0)$, the edges $v_nu_{n,i}$ can be colored using $a$ and $b$ such that $c^*(v_n) = f(v_n)$, but the coloring may be different when $c(v_nv_1) = a$ or $c(v_nv_1) = b$.

If $C$ is not $d$-uniform for any $d$, then the cycle partitions into disjoint paths $P_1,P_2,\dots,P_t$ where each $P_i$ is a maximal consecutive list of vertices on the cycle of the same degree.
The only edges incident to two paths are the edges spanning endpoints of consecutive paths.
Also, since $C$ does not contain adjacent 2-vertices, a path containing vertices of degree two has only one vertex.
There exists some $i$ where $P_i$ is either a single vertex of degree two or is a path of vertices of degree $d \geq 4$.
Starting at $v_1$, we can iteratively use Lemmas~\ref{lma:3uniformpath} and \ref{lma:4plusuniformpath} to extend the coloring $c$ to a color-blind distinguishing coloring of the edges incident to the paths surrounding $v_1$,  until both endpoints of $P_i$ are colored.
If $P_i$ is a single vertex, then $c$ is a color-blind distinguishing coloring of $C$.
Otherwise, the vertices in $P_i$ have degree $d \geq 4$ and there are at least two vertices.
We can arbitrarily extend $c$ to the first vertex of $P_i$ and then Lemma~\ref{lma:4plusuniformpath} demonstrates there is an extension of $c$ to $P_i$ such that $c$ is color-blind distinguishing.
\end{proof}

Using the above lemmas, we can extend colorings through a cactus depending on whether a block is isomorphic to $K_2$ (Lemma~\ref{lma:doublestar}) or is a cycle (Lemma~\ref{lma:hairy}).
Theorem~\ref{thm:cacti3} follows quickly from the following strengthened statement.

\begin{theorem}\label{thm:cacti3strong}
Let $G$ be a cactus where there exists a vertex $v$ of degree at least three.
If $c$ is a 3-edge-coloring of the edges incident to $v$, then there is an extension of $c$ to the edges of $G$ such that $c$ is color-blind distinguishing.
\end{theorem}

The proof of Theorem~\ref{thm:cacti3strong} almost exactly the same as the proof below of Theorem~\ref{thm:cacti2strong}, except that we use three colors and do not need to be concerned about 2-vertices and 3-uniform cycles.
The theorem below implies Theorem~\ref{thm:cacti2}.

\begin{theorem}\label{thm:cacti2strong}
Let $G$ be a cactus where $G$ has at least two edges, is not a cycle, and does not contain a 3-uniform odd cycle, and let $v$ be a non-leaf vertex in $G$.
If $c$ is a $2$-coloring of the edges incident to $v$, then there is an extension of $c$ to the edges in $G$ such that $c$ is color-blind distinguishing.
\end{theorem}

\begin{proof}
Suppose, for the sake of contradiction, that there exists a cactus $G$, vertex $v$, and coloring $c$ that satisfy the hypotheses of the theorem but there is no extension of $c$ to a color-blind distinguishing 2-edge-coloring of $G$;
select such a triple $(G, v, c)$ to minimize the number of vertices in $G$.
Note that if $v$ is the only non-leaf vertex in $G$, then the coloring $c$ is a color-blind distinguishing coloring of $G$, so $G$ has at least two non-leaf vertices, and $v$ has a non-leaf neighbor $u$.

Let $S$ be a set of vertices in $G$.
An \emph{$S$-lobe} is a subgraph of $G$ induced by $S$ and a connected component of $G - S$.
An \emph{extended $S$-lobe} is a subgraph of $G$ induced by $S$, the neighborhood of $S$, and a non-trivial connected component of $G - S$.

\begin{claim}
There is exactly one extended $\{v\}$-lobe in $G$.
\end{claim}

\begin{proof}
Suppose the extended $\{v\}$-lobes of $G$ are listed as $S_1,\dots,S_t$ with $t \geq 2$.
Each extended lobe  $S_i$ is a cactus of strictly smaller order than $G$ and $c$ can be independently extended to a color-blind distinguishing 2-edge-coloring in $S_i$.
The union of the colorings on $S_1, \dots, S_t$ is a color-blind distinguishing 2-edge-coloring of $G$, a contradiction.
\end{proof}

Hence, at most one block $B$ contains $v$ and the non-leaf vertex $u$.

If this block $B$ is a cut-edge, then the extended block $\overline{B}$ is a double-star.
By Lemma~\ref{lma:doublestar}, the coloring $c$ extends to the edges incident to $u$ such that $c^*(v) \neq c^*(u)$.
Then, the extended $\{u\}$-lobe $G'$ where $v$ is a leaf has order strictly less than the order of $G$.
Thus, $c$ extends in $G'$ to a color-blind distinguishing 2-edge-coloring of $G'$ and with the colors incident to $v$ forms a color-blind distinguishing 2-edge-coloring of $G$, a contradiction.

Thus, the block $B$ is a cycle.
Observe that $\overline{B}$ is a hairy cycle, and $B$ is not a 3-uniform cycle.
Since $\overline{B}$ has no adjacent 2-vertices, then by Lemma~\ref{lma:hairy} the coloring $c$ extends to the edges in $\overline{B}$ such that $c^*$ is a proper vertex coloring on the cycle.
Then, for every vertex $u \in V(B) \setminus \{v\}$, the extended $\{u\}$-lobe $G_u$ not containing all of $B$ is a cactus of strictly smaller order, so the coloring $c$ extends to a color-blind distinguishing 2-edge-coloring of $G_u$.
The union of these colorings agree on $\overline{B}$ and form a color-blind distinguishing 2-edge-coloring of $G$.
\end{proof}

\end{document}